\documentclass[11pt]{article}
\usepackage{calc}
\usepackage{dbnsymb} 
\usepackage{xcolor} 
\usepackage{longtable}

\usepackage{graphicx}
\usepackage{amsmath,amsthm,amsfonts,latexsym} 
\begin{document}

\newtheorem*{theo}{Theorem}
\newtheorem*{pro} {Proposition}
\newtheorem*{cor} {Corollary}
\newtheorem*{lem} {Lemma}
\newtheorem{theorem}{Theorem}[section]
\newtheorem{corollary}[theorem]{Corollary}
\newtheorem{lemma}[theorem]{Lemma}
\newtheorem{proposition}[theorem]{Proposition}
\newtheorem{conjecture}[theorem]{Conjecture}
\newtheorem{question}{Question}

\theoremstyle{definition}
 \newtheorem{definition}[theorem]{Definition} 
  \newtheorem{example}[theorem]{Example}
   \newtheorem{remark}[theorem]{Remark}
   
\newcommand{\Naturali}{{\mathbb{N}}}
\newcommand{\Reali}{{\mathbb{R}}}
\newcommand{\Complessi}{{\mathbb{C}}}
\newcommand{\Toro}{{\mathbb{T}}}
\newcommand{\Relativi}{{\mathbb{Z}}}
\newcommand{\HH}{\mathfrak H}
\newcommand{\KK}{\mathfrak K}
\newcommand{\LL}{\mathfrak L}
\newcommand{\as}{\ast_{\sigma}}
\newcommand{\tn}{\vert\hspace{-.3mm}\vert\hspace{-.3mm}\vert}
\def\A{{\cal A}}
\def\B{{\cal B}}
\def\E{{\cal E}}
\def\F{{\cal F}}
\def\G{{\cal G}}
\def\H{{\cal H}}
\def\K{{\cal K}}
\def\L{{\cal L}}
\def\N{{\cal N}}
\def\M{{\mathcal M}}
\def\gM{{\frak M}}
\def\O{{\cal O}}
\def\P{{\cal P}}
\def\S{{\cal S}}
\def\T{{\mathcal T}}
\def\U{{\cal U}}
\def\V{{\mathcal V}}
\def\qed{\hfill$\square$}

\title{Graph polynomials and link invariants as positive type functions on Thompson's group $F$}

\author{Valeriano Aiello, Roberto Conti\\}
\date{\today}
\maketitle
\markboth{V. Aiello, R. Conti}{
}
\renewcommand{\sectionmark}[1]{}
\begin{abstract}
In a recent paper Jones introduced a correspondence between elements of the Thompson group $F$ and certain graphs/links. It follows from his work that several 
polynomial invariants of links, such as the Kauffman bracket, can be reinterpreted as  coefficients of certain unitary representations of $F$.  We give a somewhat different and elementary proof of this fact for the Kauffman bracket evaluated at certain roots of unity by means of a statistical mechanics model interpretation.  Moreover, by similar methods we show that, for some particular specializations of the variables, other familiar link invariants and graph polynomials, namely the number of $N$-colourings and the Tutte polynomial, can be viewed as  positive definite functions on $F$. 

\end{abstract}

\newpage

\section{Introduction}
The Thompson group $F$ is a very 
fascinating mathematical object \cite{Cannon}. Despite the many various attempts,
after many years the question whether it is amenable remains unsettled (see e.g. \cite{Haagerup}). 
It is well known that the elements of $F$ admit some nice pictorial description in terms of pairs of rooted planar binary trees with the same number of leaves.
In a very recent paper \cite{Jo}, devoted to the search of connections between (the planar algebra description of) subfactors and the (vertex algebra description of) conformal field theory,
V. Jones gave some new insight on the mathematical structure of $F$ by considering it as a suitable replacement of the group of diffeomorphisms of the circle,
and examining its role in the possible constructions of conformal models. 
Proceeding along these lines,
and analyzing carefully the mathematical setup,
perhaps the most surprising conclusion in \cite{Jo} is that it is possible to take $F$ as a replacement of the
familiar braid groups in a description of knots and links, thus opening a new line of research
on the interplay between $F$ and low-dimensional topology. Indeed, one can associate to each element of $F$ a suitable graph and a related link diagram in a way that we briefly recall below. Especially, by exploiting this construction, 
in \cite{Jo} 
it is provided evidence that certain well-known polynomial invariants, suitably normalized, define positive definite functions of $F$
and thus give rise to unitary representations, at least for all the values of the relevant parameters corresponding to the discrete part of the Jones spectrum. This is further illustrated with  concrete examples by looking at the chromatic polynomial \cite[Proposition 5.2.1]{Jo}, the Kauffman bracket and the Jones polynomial (Section 5.3 in the aforesaid paper).

In the present paper we take up this point of view, although in a somewhat different fashion. 
We give an elementary proof 
of the 
fact that, up to a normalization, the evaluation of the Kauffman bracket at some specific roots of unity provides 
a positive definite function on $F$ by exploiting the familiar connection between link invariants and statistical mechanical models. 
We also discuss along similar lines analogous statements for the Tutte polynomial and the $N$-colouring of links. Therefore, some particular specializations of these functions also yield 
unitary representations of the Thompson group $F$.

Finally, we observe that we only consider the case of invariants of unoriented link diagrams. The case of oriented links 
is partially treated in \cite{AC} by using similar statistical mechanics arguments and, in full generality, in \cite{ACJ} by using the recent machinery developed in \cite{Jo-No}.

\section{Preliminaries}

\subsection{The Thompson group $F$}
The Thompson group $F$ can be defined by the following finite presentation
$$
	\langle x_0, x_1 \ | \ x_2 x_1 = x_1 x_3, x_3 x_1 = x_1 x_4 \rangle , 
$$
where $x_n\doteq x_0^{1-n}x_1x_0^{n-1}$ for $n\geq 2$.
In an alternative picture, $F$ can be seen as a particular subgroup of the group of homeomorphisms of the interval $[0,1]$. Indeed, it is generated by the following homeomorphisms (see \cite{Cannon} for further details)
\begin{eqnarray*}
	x_0=\left\{\begin{array}{ll} 
 		2t  & 0\leq t\leq \frac{1}{4} \\[0.5ex]
 		t+\frac{1}{4} & \frac{1}{4} \leq t\leq \frac{1}{2} \\[0.5ex]
 		\frac{t}{2}+\frac{1}{2} & \frac{1}{2} \leq t\leq 1
 		\end{array}\right. 
 		\quad 
 		x_1=\left\{\begin{array}{ll} 
 		t  & 0\leq t\leq \frac{1}{2} \\[0.5ex]
 		2t-\frac{1}{2}  & \frac{1}{2} \leq t\leq \frac{5}{8} \\[0.5ex]
 		t+\frac{1}{8} & \frac{5}{8} \leq t\leq \frac{3}{4} \\[0.5ex]
 		\frac{t}{2}+\frac{1}{2} & \frac{3}{4} \leq t\leq 1
 		\end{array}\right.
\end{eqnarray*}
An equivalent description is the following, \cite{Belk}. One can define standard dyadic intervals, namely those whose endpoints are $\frac{k}{2^n}$ and $\frac{k+1}{2^n}$ for $n, k\in\mathbb{N}$. Any finite partition of the interval $[0,1]$ made with standard dyadic intervals is called a dyadic subdivision. Given two dyadic subdivisions $\mathcal{A}$ and $\mathcal{B}$ with the same cardinality, it is possible to define a homeomorphism $f_{\mathcal{A},\mathcal{B}}:[0,1]\to [0,1]$ which maps linearly  each interval of $\mathcal{A}$ onto the corresponding interval of $\mathcal{B}$. The maps $f_{\mathcal{A},\mathcal{B}}$ form the group  $F$.
This characterization of the Thompson group has the following graphical description. Set $\T \doteq \bigcup_n \T_n$ the space of rooted planar binary trees. With $T \in \T$, we denote by $\partial T = \{f_1,\ldots,f_n\}$ the set of leaves of $T$. Of course, $\T_n \doteq \{ T \in \T \ | \ |\partial T| = n \}$. Denote by $\T^2_\partial \doteq \T \times_\partial \T$ the set of matched pairs of trees $(T_+,T_-)$, i.e. such that $\partial T_+ = \partial T_-$.
We also say that any such pair $(T_+,T_-)$ is bifurcating. To any leaf of a binary tree it is associated a standard dyadic  interval (see \cite{Belk}), thus a pair of trees can be used to determine an element of $F$. Therefore, there is a map $\T^2_\partial \to F$, $(T_+,T_-) \mapsto g(T_+,T_-)$. Indeed, it is surjective but not injective. 
A cheap way to see this is to realize that any such pair with $T_+ = T_-$ gives rise to the identity element of $F$. Moreover, it holds $g(T_-,T_+) = g(T_+,T_-)^{-1}$ and if $g(T_+,T_-) = g(T'_+,T'_-)$ it is possible to connect the two pairs by a sequence of addition/deletion of opposite carets.

\subsection{Jones' correspondence between the Thompson group and links}
In this section we review Jones'  procedure that associates links to elements of $F$, \cite{Jo}. This can be described by the composition of four maps. Denote by $\G$ the set of signed planar oriented finite graphs. First of all, there is a specific function $\Gamma: \T^2_\partial \to \G$, $(T_+,T_-) \mapsto \Gamma(T_+,T_-)$. For any $\Gamma \in \G$ one has the associated medial graph $\Phi(\Gamma) \in \M$ (see \cite{Godsil}, p.398), that is a $4$-regular finite graph. Finally, there is a map $L: \M \to \L$, where $\L$ denotes the set of unoriented link diagrams.
By composition, get a map $\T^2_\partial \to \L$ given by 
$$(T_+,T_-) \mapsto L(\Phi(\Gamma(T_+,T_-))) \ , $$
also denoted $L(T_+,T_-)$ for simplicity. Every element in the group $F$ has a unique reduced,
that is without opposing carets, tree diagram (see \cite{Belk}, p. 6). By this result, there is 
a canonical section $F \to \T^2_\partial$ and by the above this result provides unambiguously a map
$$F \to \L.$$
We denote it simply by $F \ni g \mapsto L_g \in \L$.

\begin{example} \label{ese-omega2}
We now show an example of the above procedure with the trefoil (cf. \emph{\rm Rem.} 5.3.3, \cite{Jo}).
Let $\omega^2=(x_1x_0^{-1})^2\in F$ be the element described by the following pair of trees, with associated oriented graph $\Gamma(\omega^2)$,  and its associated medial graph $\Phi(\Gamma(\omega^2))$ 
$$
\includegraphics[scale=.5]{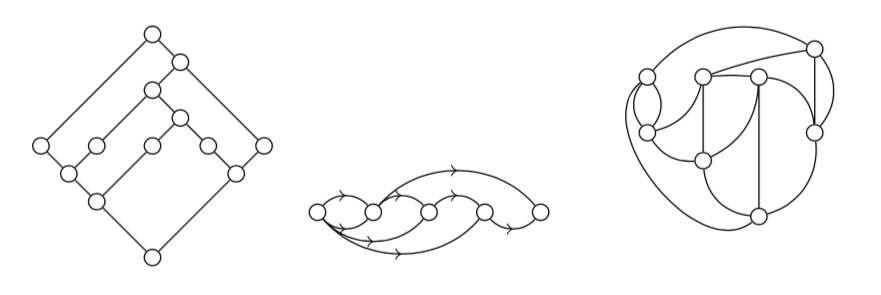}
$$

The associated link diagram is obtained from  $\Phi(\Gamma(\omega^2))$ by putting a crossing of type $\slashoverback$ over the  four vertices corresponding to the upper tree, and a crossing of type  $\backoverslash$ on the remaining vertices, as follows:
$$
\includegraphics[scale=.3]{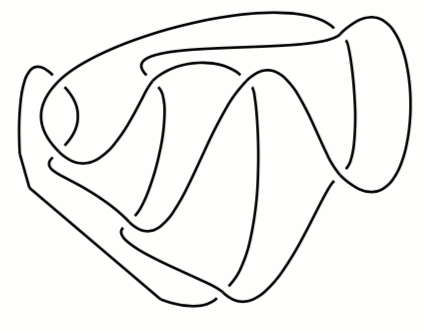}
$$
It is easy to see that the diagram $L_{\omega^2}$ obtained in this way represents a trefoil.
\end{example}

\noindent 
The signed graph $\Gamma(T_+,T_-)$ is obtained by gluing along the common boundary the
two graphs $\Gamma_+(T_+)$ and $\Gamma_-(T_-)$.
The edges of $\Gamma_+(T_+)$ are signed with + and those of $\Gamma_-(T_-)$ with - \ , respectively . 

\begin{remark}
	(\cite{Jo}, p. 19-20) The graphs $\Gamma_+(T_+)$ and $\Gamma_-(T_-)$ are rooted trees (not bifurcating in general). Let $\Psi$ be a rooted tree, then there exists a bifurcating tree $T_\pm$ such that $\Psi=\Gamma_\pm(T_\pm)$ (see {\it loc.cit.}, \emph{Lemma} 4.1.1.  for a proof). Thus, giving a pair of matched trees $(T_+,T_-)$ is equivalent to giving the graph $\Gamma(T_+,T_-)$.
\end{remark}

\begin{remark}
Actually $\Gamma_+(T_+)$ is nothing but a subgraph of the dual of the tree $T_+$ considered as a planar graph in the upper half-plane with the leaves on the boundary. It is the subgraph generated by the edges corresponding to the north-east edges of $T_+$. Similarly, $\Gamma_-(T_-)$ is a subtree of the dual of $T_-$ in the lower half-plane. 
\end{remark}

\begin{proposition} \label{link-inverse-element}
For any $g \in F$, it holds
$$L_{g^{-1}} = (L_g)^* \ , $$
where $^*$ denotes the mirror image link.
\end{proposition}

\begin{proof}
If we associate to $g$ the pair of trees $(T_+,T_-)$, then we associate to $g^{-1}$ the pair $(T_-,T_+)$. When we consider the medial graph of $\Gamma_-(T_-)$ (for $g$) we used the crossing $\backoverslash$. When we consider the medial graph of $\Gamma_+(T_-)$ (for $g^{-1}$) we use the crossing $\slashoverback$. If we rotate the link $L_{g^{-1}}$ about the x-axis we recognise the mirror image of the original link. 
\end{proof}

We record a simple but useful result about cancelling carets for later reference. Since this fact is already mentioned in \cite[p.19]{Jo} we leave the proof to the reader.
\begin{proposition}  \label{add-opp-car}
Let $(T_+,T_-)$ be a pair of bifurcating trees, and consider another such pair $(T'_+,T'_-)$ obtained 
by adding an opposing pair of carets. Then 
\begin{enumerate}
\item $\Gamma(T'_+,T'_-)$ differs from $\Gamma(T_+,T_-)$ by the addition of a new $2-$valent vertex only connected to a vertex of $\Gamma(T_+,T_-)$ on its left;
\item $L(T'_+,T'_-) = L(T_+,T_-) \cup O$ (addition of a distant unknot).
\end{enumerate}
\end{proposition}

\section{Positive type functions and polynomial link invariants, graph polynomials}
The aim of this paper is to provide an answer to the following question: for which (polynomial) invariants of unoriented links (or graph polynomial) $P$ is the function $F \ni g \to P(L_g)$ (or $F \ni g \to P(\Gamma(g))$) of positive type? 
We recall the definition of positive type function on a discrete group $G$.
\begin{definition}
A function $\phi: G\rightarrow
\mathbb{C}$ is said {\it  of positive type}  (or {\it positive definite}) if one of the following equivalent conditions hold
\begin{itemize}
\item[1)] $\sum_{i, j= 1,\ldots , r} a_i\overline{a_j}\phi (g_ig_j^{-1})\geq 0$ for any $r$ and $g_1,\ldots,g_r\in G$, $a_1,\ldots,a_r\in \mathbb{C}$;
\item[2)] the matrix $(\phi(g_ig_j^{-1}))_{i,j=1,\cdots , r}\geq 0$  (i.e. it is positive semidefinite) for any $r$ and $g_1,\ldots,g_r\in G$. 
\end{itemize}
\end{definition}

\begin{remark} \label{fpt}
	The polynomial invariant must satisfy the following conditions (see \cite{F}  for more details):
	\begin{enumerate}
		\item $P(e)\in \mathbb{R}^+$: it is enough consider the definition with $r=1$.
		\item $P(e)\geq |P(g)|$  $\forall g\in G$: it is enough to consider the definition with $r=2$, $g_1=e$ and $g_2=g$. The property follows because the matrix is positive definite and the determinant is positive. 
		\item $P(x^{-1})=\overline{P(x)}$: it is enough to consider the definition with $n=2$, the property follows becuase the matrix is selfadjoint.
	\end{enumerate}
\end{remark}

We give an example of a graph polynomial that is positive definite function on $F$, namely the (normalized) chromatic polynomial. As a consequence of \cite[Proposition 5.2.1.]{Jo}, Jones shows  that 
$(Q-1)^{-n} {\rm Chr}_{\Gamma(T_+,T_-)}(Q)$ is a function of positive type on $F$, for all values of $Q$, using the theory of planar algebras. 
We recall the definition of the chromatic polynomial and present a detailed proof of this fact when $Q\in \mathbb{N}$.
The method employed here will be useful later, when considering other invariants.\\
Let $G$ be a graph. 
The chromatic polynomial of $G$ is the polynomial ${\rm Chr}(G,t)$ such that, for any $Q \in {\mathbb N}$,
${\rm Chr}(G,Q)$ is the number of proper colourings of the vertices of $G$ with $Q$ colours.
Moreover, if $e\in E(G)$, it satisfies the following condition
\begin{displaymath}
    {\rm Chr}(G,Q)=\left\{\begin{array}{ll} 
 		0  & \textrm{ if $e$ is a loop } \\
 		{\rm Chr}(G-e, Q)- {\rm Chr}(G/e,Q) &  \textrm{ otherwise.}\\
 		\end{array}\right.
\end{displaymath}
\begin{example}
 Consider the following $\Gamma$-graphs 
$$
\includegraphics[scale=.4]{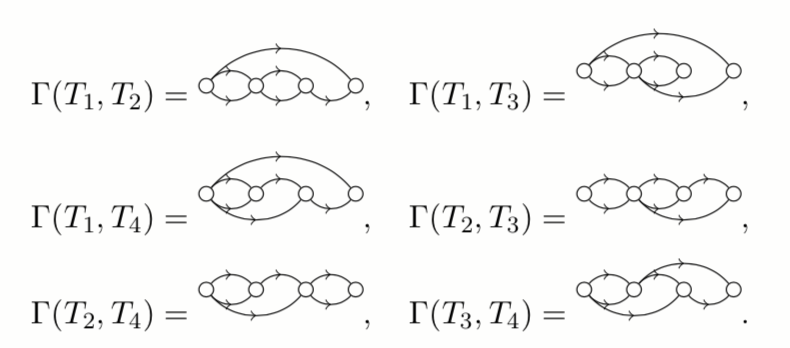}
$$
When $Q=4$, the chromatic polynomial takes the following values on the above graphs: 
\begin{itemize}
	\item ${\rm Chr}(\Gamma(T_1,T_2))=Q(Q-1)(Q-1)(Q-1)-Q(Q-1)(Q-2)=36+48=84$,	
	\item ${\rm Chr}(\Gamma(T_1,T_3))=Q(Q-1)(Q-1)(Q-2)=72$,
	\item ${\rm Chr}(\Gamma(T_1,T_4))=Q(Q-1)(Q-1)(Q-2)-Q(Q-1)(Q-2)=48$
	\item ${\rm Chr}(\Gamma(T_2,T_3))=Q(Q-1)(Q-1)(Q-2)=72$
	\item ${\rm Chr}(\Gamma(T_2,T_4))=Q(Q-1)(Q-1)(Q-2)=72$
	\item ${\rm Chr}(\Gamma(T_3,T_4))=Q(Q-1)(Q-2)(Q-2)=48$
\end{itemize}
The associated matrix is
\begin{displaymath}
    A=\frac{1}{27}\left(\begin{array}{cccc} 
 		108 & 84 & 72 & 48 \\
 		84 & 108 & 72 & 72 \\
 		72 & 72 & 108 & 48 \\
 		48 & 72 & 48 & 108
 	\end{array}\right).
\end{displaymath}
\end{example}

Let $g\in F$ and suppose that $g=g(T_+,T_-)$ for some $T_+, T_-\in \T_n$. Then we can consider the well defined function on $F$: ${\rm Chr}(g,Q)={\rm Chr}(\Gamma(T_+,T_-),Q)(Q-1)^{-n+1)}$.
Let $g_1, \ldots, g_r \in F$ and let $g_i = g(T^i_+,T^i_-)$ (not necessarily in reduced form).
Without loss of generality, inserting pairs of opposing carets whenever necessary, we can assume that $T^i_- = T^j_-$, for all $i,j=1,\ldots,r$. Then $g_i g_j^{-1} = g(T^i_+,T^j_+)$ and
in order to prove that ${\rm Chr}(g,Q)$ is of positive type, we have to consider the symmetric matrix 
$$\Big({\rm Chr}(\Gamma(T^i_+,T^j_+),Q)/(Q-1)^{n-1}\Big)_{i,j=1}^r \ . $$ 
Recall that both $\Gamma_+(T_+)$ and $\Gamma_-(T_-)$ are rooted trees with $n$ vertices, and that for any such tree the chromatic polynomial evaluated at $Q$ takes the value $Q(Q-1)^{n-1}$.
It is clear that, for the diagonal terms, 
${\rm Chr}(\Gamma(T^i_+,T^i_+),Q) = {\rm Chr}(\Gamma_+(T^i_+),Q) = Q(Q-1)^{n-1}$.
Before proceeding with the proof, we introduce some notations and preliminary results. Consider two trees $T_i, T_j\in\T_n$. We denote the set of colouring of $\Gamma_+(T_i)$ and $\Gamma(T_i,T_j)$  by ${\rm Col}(T_i)$ and ${\rm Col}(T_i,T_j)$, respectively. Any colouring $\alpha\in {\rm Col}(T_i)$ can be described by an array, which we denote by the same symbol, $\alpha=(i_1,i_2,\cdots,i_n)$, with $i_k\in \{1,\cdots, Q\}$. Consider the Hilbert space $\H_n=\bigotimes_{i=1}^n \mathbb{C}^Q$. To any $\alpha\in {\rm Col}(T_i)$ we can associate the vector $v_\alpha=e_{i_1}\otimes\cdots \otimes e_{i_n}\in \H_n$. We also define the vector $v_{T_i} = \sum_{\alpha\in {\rm Col}(T_i)} v_\alpha \in \H_n$. 
Let $T_i, T_j\in\T_n$. As consequence of the orthogonality of the vectors $v_\alpha$, then the following equality holds
$$
	{\rm Chr}(\Gamma(T_i,T_j))=\langle v_{T_i}, v_{T_j} \rangle 
= \left\langle \sum_{\alpha\in {\rm Col}(T_i)} v_\alpha, \sum_{\beta\in {\rm Col}(T_j)} v_\beta\right\rangle.
$$ 
In fact, the right hand side of the formula counts the number of colourings that are valid for both $\Gamma_+(T_i)$ and $\Gamma_-(T_j)$, i.e. the colourings of $\Gamma(T_i,T_j)$.
\begin{lemma} (\cite{Bhatia}, p. 2, 3)\label{inner-prod-matrix}
	$A\in M_r(\mathbb{C})$ is positive semidefinite if and only if there exist a Hilbert space $H$ and vectors $v_1,\cdots, v_r\in H$ such that
	$$
		(a_{ij})=(\langle v_i,v_j\rangle).
	$$
\end{lemma}
By the above discussion we have that the function ${\rm Chr}(g,Q)$ is a function of positive type on $F$.

\section{The Tutte polynomial}
We briefly recall some definitions, the interested reader is referred to \cite{Bollobas} for further details.
Let $G$ be a graph. We denote the vertices by $V(G)$ and the edges by $E(G)$, or for simplicity by $V$ and $E$, respectively. The Tutte polynomial may be defined as
$$
	T_G(x,y)=\sum_{S\subset E(G)} (x-1)^{c(S)-c(E(G))}(y-1)^{c(S)+|S|-|V(G)|}
$$
where $c(S)$ denotes the number of connected components of the graph generated by $S$. The Tutte polynomial satisfies the following deletion and contraction rule
\begin{displaymath}
	T_G=\left\{
 		\begin{array}{ll}
 		xT_{G-e}=xT_{G/e} & \textrm{ if $e$ is a bridge} \\ 
 		yT_{G/e} & \textrm{ if $e$ is a loop} \\ 
 		T_{G-e}+T_{G/e} & \textrm{ otherwise.} \\ 
 		\end{array}
 		\right.
\end{displaymath}
We want to define a function on the Thompson group $F$ using the Tutte polynomial. First of all we need the following lemma
\begin{lemma}
Consider $T_1, T_2\in\T_n$ and let $T'_1, T'_2\in\T_{n+1}$  be the threes obtained by the addition of a pair of opposite carets. Then the following equality holds
$$
	T_{\Gamma(T'_1,T'_2)}(x,y)=T_{\Gamma(T_1,T_2)}(x,y) (x+y).
$$
\end{lemma} 
\begin{proof}
By \emph{Proposition} \ref{add-opp-car}, the graphs $\Gamma(T_1,T_2)$ and $\Gamma(T'_1,T'_2)$ differs by a pair of parallel edges attached only to a vertex. We denote these two edges by $e_+$ and $e_-$. Using the rule of deletion and contraction we have that
\begin{align*}
	T_{\Gamma(T'_1,T'_2)}(x,y)&=T_{\Gamma(T'_1,T'_2)-e_-}(x,y)+T_{\Gamma(T'_1,T'_2)/e_-}(x,y)=\\
	&=T_{\Gamma(T'_1,T'_2)-e_-}(x,y)+y \, T_{\Gamma(T_1,T_2)}(x,y)=\, \textrm{ ($e_+$ is a loop in $\Gamma(T'_1,T'_2)/e_-$)}\\
	&=x \, T_{\Gamma(T_1,T_2)}(x,y)+y \, T_{\Gamma(T_1,T_2)}(x,y)=\, \textrm{ ($e_+$ is a bridge in $\Gamma(T'_1,T'_2)-e_-$)}\\
	&=T_{\Gamma(T_1,T_2)}(x,y) (x+y).
\end{align*}
\end{proof}

We define the Tutte function as
$$
	T_g(x,y)\doteq T_{\Gamma(T_1,T_2)}(x,y) (x+y)^{-n+1} 
$$
where $T_1, T_2\in\T_n$ and $g=g(T_1,T_2)\in F$. The above lemma shows that it is well defined.

Consider the partition function of the Potts model 
$$
	Z(G;Q,K)=\sum_\sigma e^{-K\sum_{\overline{ij}\in E(G)} (1-\delta(\sigma_i,\sigma_j))}
$$
where $\overline{ij}$ is an edge and $\sigma_i$ is the spin at site $i\in V(G)$.
It may be proved that (up to a constant) the partition function is equal to the Tutte polynomial for certain values of the variables. More precisely
$$
	Z(G;Q,K)=Q(y-1)^{|V|-1}y^{-|E|}T_G(x,y)
$$
for $y=e^K$, $x=\frac{y+Q-1}{y-1}$. For more details we refer to \cite{Welsh}, p. 1127, 1132.

\begin{theorem}
The function $T_g(x,y)$ is of positive type on $F$ for $y=e^K$, $x=\frac{y+Q-1}{y-1}$ and $K\neq 0$.
\end{theorem}
\begin{proof}
Without  loss of generality we can suppose that  $g_i=g(T_+^i,T_-^i)$  with $T_\pm^i\in \T_n$ for $i=1, \cdots, r$, 
and $g_ig_j^{-1}=g(T^i_+,T^j_+)$. Therefore, we need to consider  the matrix $\Big(T_{\Gamma(T^i_+,T^j_+)}(x,y)/(x+y)^{n-1}\Big)_{i,j=1}^r$
and show that it is positive semidefinite. 
Our aim is to use the Potts partition function and prove the claim by actually showing 
that $(Z(\Gamma(T^i_+,T^j_+);Q,K))_{i,j=1}^r$ is positive semidefinite. Indeed, the partition function and the Tutte function differ by the four different factors: $Q$,  $y^{-|E|}$, $(y-1)^{|V|-1}$ and $(x+y)^{|V|-1}$. We notice that all the factors do not depend on the pair $(i,j)$. Moreover, $Q>0$ and $y^{-|E|}>0$ so we may neglect them. The remaining factor, $(y-1)^{|V|-1}(x+y)^{|V|-1}$, is positive as long as $K\neq 0$.

Suppose that $G=\Gamma(T_+,T_-)$. We note that the partition function may be rewritten as
$$
		Z(\Gamma(T_+,T_-);Q,K)=\sum_\sigma e^{-K\sum_{\overline{ij}\in E(\Gamma_+(T_+))} (1-\delta(\sigma_i,\sigma_j))}e^{-K\sum_{\overline{ij}\in E(\Gamma_-(T_-))} (1-\delta(\sigma_i,\sigma_j))}.
$$
For any $\sigma=(\sigma_1,\cdots ,\sigma_n)$, the expression $e^{-K\sum_{\overline{ij}\in E(\Gamma_+(T_+))} (1-\delta(\sigma_i,\sigma_j))}$ defines the $\sigma-$th component of a vector in $\H=\mathbb{C}^{Q^n}$, i.e the component corresponding to $e_{\sigma_1}\otimes \cdots\otimes e_{\sigma_n}$. Thus, we may define a vector $v_{T_i}$ and similarly a vector $v_{T_j}$ such that 
$$
	Z(\Gamma(T_i,T_j),x,y)=\langle v_{T_i},v_{T_j}\rangle.
$$
It follows by \emph{Lemma} \ref{inner-prod-matrix} that the matrix $(T_{g_ig_j^{-1}}(x,y))_{i,j=1}^r$ is positive semidefinite for any $r$, i.e. the function $T_g(x,y)$ is of positive type.
\end{proof}

One might wonder whether there are other values of the parameters for which the Tutte function is of positive type.

\begin{remark}
We want to observe that the result on the chromatic polynomial seen in the previous section can be obtained by results in the present section. In fact, consider the equality 
$$
	y^{n}Z(G;Q,K)=Q(y-1)^{n-1}T_G(x,y)
$$
where $y=e^K$, $x=\frac{y+Q-1}{y-1}$. For $K\to -\infty$, the right hand-side converges to 
$$
	Q(-1)^{n-1}T_G(1-Q,0)
$$
It is well known that this limit is equal to ${\rm Chr}(G,Q)$. We now consider the left hand-side. First of all, we observe that the integer $\sum (1-\delta(\sigma_i,\sigma_j))$ can only take values between $0$ and $n$. If the sum corresponding to $\sigma$ is smaller than $n$, then $e^{K(n-\sum \cdot)}$ vanishes in the limit when $K\to -\infty$. If the sum is equal to $n$, then the exponential is equal to $1$. It may easily be  seen that  the vectors considered for the chromatic polynomial correspond to those chosen for the Tutte polynomial (up to a factor $e^{Kn}$) in the limit $K\to -\infty$. 
\end{remark}

\section{The Kauffman bracket}
In \cite{Jo} Jones proved that certain evaluations of the Kauffman bracket give rise to positive definite functions. Our aim is to give a different and elementary proof of this fact, at least for some specific values of the variable.
We need to recall some preliminary definitions and results. In the first place, the Kauffman bracket \cite{Kauffman} is defined by the following skein-relation
$$
		 \left\{\begin{array}{l}
 		\langle \slashoverback\rangle=A\langle \smoothing \rangle+A^{-1}\langle \hsmoothing \rangle\\
 		\langle O\rangle= 1.
 		\end{array}
 		\right.
$$
This polynomial is invariant under Reidemeister moves of type 2 and 3. The first Reidemeister move produces a factor $(-A^3)^{\pm 1}$ (see \cite{Bollobas}, p. 365).
\smallskip

To begin with, we prove a  lemma that will allow us to define a function on $F$ associated with the Kauffman bracket. 
\begin{lemma}
Consider $T_1, T_2\in\T_n$ and let $T'_1, T'_2\in\T_{n+1}$  be the trees obtained by the addition of a pair of opposite carets. Then the following equality holds
$$
	\langle L(T'_1,T'_2)\rangle (A)=(-A^2-A^{-2})\langle L(T_1,T_2)\rangle (A).
$$
\end{lemma} 
\begin{proof}
By \emph{Proposition} \ref{add-opp-car}, the graphs $\Gamma(T_1,T_2)$ and $\Gamma(T'_1,T'_2)$ differs by a pair of parallel edges attached only to a vertex, this means that  a new distant unknot has been added. The properties of the Kauffman bracket imply the claim.
\end{proof}

Now the 
 Kauffman bracket function can be defined as
$$
	\langle g\rangle (A)\doteq (-A^2-A^{-2})^{-n}\langle L(T_1,T_2)\rangle (A)
$$
where $T_1, T_2\in\T_n$ and $g=g(T_+,T_-)\in F$. The above lemma shows that the function is well defined.

We now recall some results proved in \cite{DLH-J}. Consider a signed graph $G$. Denote by $G^+$ and $G^-$ the subgraphs whose edges are the positive and the negative edges, respectively. For any $i, j\in V(G)$, define the function
\begin{eqnarray*}
	w(\sigma_i,\sigma_j)&=& \left\{\begin{array}{cc}
 		-A^3 & \textrm{ if $\sigma_i=\sigma_j$} \\ 
 		A^{-1} & \textrm{ if $\sigma_i\neq \sigma_j$} \\ 
 		\end{array}
 		\right.
\end{eqnarray*}
where $\sigma_i$ is the spin at site $i$. Set $w_+(\sigma_i,\sigma_j)=w(\sigma_i,\sigma_j)$ and $w_-(\sigma_i,\sigma_j)=w(\sigma_i,\sigma_j)^{-1}$. 

Consider the partition function defined by 
$$
	Z_{G}(A)=\left(\frac{1}{\sqrt{Q}}\right)^{|V(G)|+1}\sum_\sigma \prod_{\overline{ij}\in E(G^+)} w_+(\sigma_i,\sigma_j)\prod_{\overline{ij}\in E(G^-)}w_-(\sigma_i,\sigma_j) \ ,
$$
where the sum over $\sigma$ runs over all the spin configurations $\{1,\cdots,Q\}^{|V(G)|}$. 

Given a link $L$ with link diagram $D$, we denote by $F(D)$ its face graph (for a definition see \cite{Godsil}, p. 379).
We notice that the face graph and $\Gamma$-graphs coincide (as signed graphs), namely
$$
	F(L(T_+,T_-))=\Gamma(T_+,T_-).
$$

\begin{lemma}
Let $Q \geq 2$ be an integer and let $A$ be such that $A^2 + A^{-2} + \sqrt Q = 0$. Consider a link $L$ with link diagram $D$.
Then  the following identity holds
$$
	\langle L(D)\rangle=Z_{F(D)}.
$$
\end{lemma}
\begin{proof}
The claim follows from the following identities
\begin{eqnarray*}
	V_{\stackrel{\to}{L}(D)}(A^{-4})&=&(-A^3)^{-wr(\stackrel{\to}{L}(D))}\langle L(D)\rangle\\
	V_{\stackrel{\to}{L}(D)}(A^{-4})&=&(-A^3)^{-wr(\stackrel{\to}{L}(D))}Z_{F(D)},	
\end{eqnarray*}
where $wr(\cdot)$ denotes the writhe.
\end{proof}

Before stating the main theorem of this section we formulate the following simple, but useful, lemma:
\begin{lemma}
Consider the 
equation 
$$
A^4+\sqrt{Q}A^2+1=0 \qquad Q\geq 2.
$$
Then 
\begin{itemize}
\item if $Q=2$, the complex solutions have modulus equal to $1$, namely they are $A=\pm e^{3 \pi i/8}$, $A=\pm e^{-3 \pi i/8}$;
\item if $Q=3$, the complex solutions have modulus equal to $1$, namely they are $A= \pm e^{5 \pi i/12}$,  $A= \pm e^{-5 \pi i/12}$;
\item if $Q=4$, the complex solutions have modulus equal to $1$, namely they are $A=\pm i$, each one with with multiplicity two;
\item if $Q\geq 5$, the complex solutions are purely imaginary with modulus different from 1.
\end{itemize}
\end{lemma}
\begin{proof}
All the statements are easily verified. We only make one comment for the case $Q\geq 5$. In this case, we have that
$$
A^2=\frac{-\sqrt{Q}\pm \sqrt{Q-4}}{2}
$$
The solutions corresponding to the negative sign clearly have modulus different from $1$. The solutions corresponding to the positive sign may be rewritten as
$$
A^2=\frac{-2}{\sqrt{Q}+ \sqrt{Q-4}}
$$
and the condition $|A|=1$ require $\sqrt{Q}+ \sqrt{Q-4}=2$. However, this happens only when $Q=4$.
\end{proof}

\begin{theorem} \label{pos-def-kauff}
The function $\langle g\rangle(A)$, where $A$ is any solution of $A^2 + A^{-2} + \sqrt Q = 0$ for $Q=2, 3, 4$, is of positive type  on $F$.
\end{theorem}
\begin{proof}
As we have done before,
without loss of generality we can suppose that the $g_i=g(T_+^i,T_-^i)$  with $T_\pm^i\in \T_n$ for $i=1, \cdots, r$ and $g_ig_j^{-1}=g(T^i_+,T^j_+)$. Therefore, it is enough to consider  $\Big(\langle L(T^i_+,T^j_+)\rangle /(-A^2-A^{-2})^{n}\Big)_{i,j=1}^r$ and prove that $(\langle L(T^i_+,T^j_+)\rangle)_{i,j=1}^r$ is positive semidefinite. Our aim is to use the partition function in order to prove our claim. The partition function and the Kauffman polynomial are equal by the above lemma. 
Therefore, recalling the equality between the face graph and the  $\Gamma$-graph, 
\begin{eqnarray*}
	\langle L(T_+, T_-)\rangle &=& Z_{F(L(T_+, T_-))}= Z_{\Gamma(T_+, T_-)}=\\
	&=&\left(\frac{1}{\sqrt{Q}}\right)^{n+1}\sum_\sigma \prod_{\overline{ij}\in E(\Gamma_+(T_+))} w_+(\sigma_i,\sigma_j)\prod_{\overline{ij}\in E(\Gamma_-(T_-))}w_-(\sigma_i,\sigma_j).
\end{eqnarray*}
For any $\sigma=(\sigma_1,\cdots ,\sigma_n)$, the expression $\prod_{\overline{ij}\in E(\Gamma_-(T_-))} w_+(\sigma_i,\sigma_j)$ defines the $\sigma$-th component of a vector in $\H=\mathbb{C}^{Q^n}$, i.e the component corresponding to $e_{\sigma_1}\otimes \cdots\otimes e_{\sigma_n}$. Thus, we may define a vector $v_{T_+^i}$ and similarly a vector $v_{T_+^j}$ such that 
$$
	Z(\Gamma(T_i,T_j))=\langle v_{T_+^i},v_{T_+^j}\rangle.
$$
It follows by \emph{Lemma} \ref{inner-prod-matrix} that the matrix $(\langle L(T^i_+,T^j_+)\rangle)_{i,j=1}^r$ is positive semidefinite for any $r$, i.e. the function $\langle g\rangle$ is of positive type.
\end{proof}

\begin{remark}
We observe that the assumption of $Q=2, 3, 4$ was made in order to recognize the summands of the partition function as the scalar product of two vectors. In fact, in these cases $A$ is a complex root of unit and 
$$
w_-(\cdot,\cdot)=w(\cdot,\cdot)^{-1}=\overline{w(\cdot,\cdot)}=\overline{w_+(\cdot,\cdot)}.
$$
\end{remark}

\begin{remark}
When one considers a signed planar graph, a knot can be obtained drawing the medial graph and using two different types of crossing depending on the sign (see \cite{Godsil}, chapters 16, 17). 
For simplicity, we have decided to follow the same notation as in Jones' paper \cite{Jo}. However, we warn the reader that  the correspondence berween signs and crossing  differs from the one adopted in standard textbooks (for example see p. 380 in \cite{Godsil} and p. 33 in \cite{Jo}).
\end{remark}

\begin{example}
Consider the following elements of the Thompson group $F$: $g_1=x_0^{-1}$, $g_2=x_1^{-1}$, $g_3=\omega=x_1x_0^{-1}$, $g_4=\omega^{-1}=x_0x_1^{-1}$.
\\
\noindent
Now we compute $\langle g_1g_2^{-1}\rangle$. The element $x_0^{-1}x_1$ can be described by the trees $T_{12}^+, T_{12}^-\in\T_5$. 
Moreover, $L(T_{12}^+, T_{12}^-)$ is equivalent to the unknot. However, in order to obtain the unknot we have to use four times the first Reidemeister move, two producing a factor $(-A^3)^2$ and the others its inverse. Therefore
\begin{eqnarray*} 
	\langle T_{12}^+,T_{12}^- \rangle &=& 1.
\end{eqnarray*}
\\
\noindent
Now we compute $\langle g_1g_3^{-1}\rangle$. The element $x_0^{-1}\omega^{-1}$ is equal to $x_1^{-1}$. In the standard (i.e. minimal) representation  $x_1^{-1}$ can be described by a pair trees $T_{13}^+, T_{13}^-\in\T_4$. 
The link associated to $x_1^{-1}$ is  equivalent to the disjoint union of two unknots. In order to obtain this trivial link one has to perform two Reidemeister moves of type $1$, one producing the factor $(-A^{-3})$ and the other one the factor $(-A^3)$. Therefore,
\begin{eqnarray*} 
	\langle T_{13}^+,T_{13}^- \rangle &=& (-A^2-A^{-2}).
\end{eqnarray*}
\\
\noindent
Now we compute $\langle g_1g_4^{-1}\rangle$. The element $g_1g_4^{-1}$ is equal to $x_0^{-1}x_1x_0^{-1}$ and it is associated to a pair of trees $T_{14}^+, T_{14}^-\in \T_5$. After an easy calculation one obtains the \emph{Hopf link}. Actually, we have to perform the resolution of two curls (one positive and one negative). Therefore, we have that 
\begin{eqnarray*} 
	\langle T_{14}^+,T_{14}^- \rangle &=& -A^{-4}-A^{4}.
\end{eqnarray*}
\\
\noindent
Now we compute $\langle g_2g_3^{-1}\rangle$. The element  $g_2g_3^{-1}$ is equal to $x_1^{-1}x_0x_1^{-1}$ and it is associated to a pair of trees $T_{23}^+, T_{23}^-\in \T_5$. The associated link is equivalent to the unknot. However, in order to obtain this trivial knot one has to perform four Reidemeister moves of type $1$, two producing the factor $(-A^{-3})^2$ and the other two the factor $(-A^3)^2$. Therefore 
\begin{eqnarray*} 
	\langle T_{23}^+,T_{23}^- \rangle &=& 1.
\end{eqnarray*}
\\
\noindent
Now we compute $\langle g_2g_4^{-1}\rangle$. The element  $g_2g_4^{-1}$ is equal to $x_0^{-1}$. In the standard (i.e. minimal) representation  $x_0^{-1}$ can be described by a pair trees $T_{24}^+, T_{24}^-\in\T_3$. 
The link associated to $x_0^{-1}$  is equivalent to the disjoint union of two unknot. Moreover, in order to obtain this trivial link one has to perform two Reidemeister moves of type $1$, one producing the factor $(-A^{-3})$ and the other one the factor $(-A^3)$. Therefore 
\begin{eqnarray*} 
	\langle T_{24}^+,T_{24}^- \rangle &=& 1.
\end{eqnarray*}
\\
\noindent
Now we compute $\langle g_3g_4^{-1}\rangle$. As we have seen in \emph{Example} \ref{ese-omega2} there exist two trees $T_{34}^+, T_{34}^-\in \T_5$ such that $\omega^2=g(T_{34}^+,T_{34}^-)$. After the application of two Reidemeister moves of type $2$, one move of type $3$ and one of type $1$ we get the trefoil. We have that
\begin{eqnarray*} 
	\langle T_{34}^+,T_{34}^-\rangle &=& (-A^{-3})(-A^{-5}-A^3+A^7)=1-A^{4}+A^{-8}.
\end{eqnarray*}
So far we have computed the entries over the diagonal. The entries under the diagonal differ from the corresponding entries over the diagonal by the change of variable $A \to A^{-1}$, cf. \emph{Proposition} \ref{link-inverse-element}. The only different entry is the one in position $(4,3)$:
$$
	\langle  T_{34}^-,T_{34}^+\rangle = 1-A^{-4}+A^{8}.
$$ 
Summing up, the matrix associated to these elements is
\begin{eqnarray*}
		\frac{1}{Q^{5/2}}\left(
 		\begin{array}{cccc}
 		Q^2 & 1 & Q & -A^4-A^{-4}\\
 		1 & Q^2 & 1 & Q\\
 		Q & 1 & Q^2 & A^{-8}-A^4+1\\ 
 		-A^4-A^{-4} & Q & A^{8}-A^{-4}+1 & Q^2\\
 		\end{array}
 		\right) .	
\end{eqnarray*}
One can check that for $A=e^{3\pi i /8}$ and $A=i$ we get the following matrices, respectively
\begin{eqnarray*}
		\frac{1}{2^{5/2}}\left(
 		\begin{array}{cccc}
 		4 & 1 & 2 & 0\\
 		1 & 4 & 1 & 2\\
 		2 & 1 & 4 & i\\ 
 		0 & 2 & -i & 4\\
 		\end{array}
 		\right) 	
 		\quad \textrm{ and } \quad
		\frac{1}{2^{5}}\left(
 		\begin{array}{cccc}
 		16 & 1 & 4 & -2\\
 		1 & 16 & 1 & 4\\
 		4 & 1 & 16 & 1\\ 
 		-2 & 4 & 1 & 16\\
 		\end{array}
 		\right), 	
\end{eqnarray*}
\end{example}

The above example also shows that the obtained matrix is not self-adjoint for $Q\geq 5$ and thus certainly not positive semidefinite. Therefore we have the following result
\begin{theorem}
The function $\langle g\rangle (A)$, where $A$ is any solution of $A^2 + A^{-2} + \sqrt Q = 0$ for $Q\geq 5$, is not positive definite on the Thompson group $F$.
\end{theorem}

The $F$-index of a link is defined as the smallest number of leaves required for an element of $F$ to give that link (\cite{Jo}, p. 41). In the previous example we showed that the Hopf link and the Trefoil can be obtained with pairs of trees in $\T_5$.  One might wonder whether their  $F$-index is $5$. The following result allow us to give a positive answer to this question.
\begin{proposition}
Every non-trivial knot/link has $F$-index at least equal to $5$.
\end{proposition}
\begin{proof}
By symmetry (i.e. taking into account \emph{Propositions} \ref{link-inverse-element} and \ref{add-opp-car}), one has to consider only the following $\Gamma$-graphs
$$
\includegraphics[scale=.5]{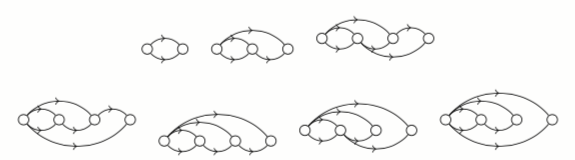}
$$

\noindent It can be easily checked that the associated links are trivial.
\end{proof}

\section{Fox $N$-colouring}
Our aim is to define a function of positive type on $F$ using the number of colourings of a link. First of all we introduce the preliminary definitions.
Let $Q$ be a positive odd integer. The definition $Q$-colouring is a generalization of tricoloring of link (see \cite{encycl}, p. 162, and \cite{P}).
\begin{definition}
Let $L$ be a link and $D$ one of its diagrams. A link is said to be $Q$-colourable if every arc of the diagram is associated to an element of $\mathbb{Z}_Q$ (called colour), such that at each crossing the sum of the colours of the undercrossing is equal to the sum of the overcrossing. The number of $Q$-colourings is a link invariant and is denoted by ${\rm Col}_Q(L)$.
\end{definition}

In order to prove that $Q$-colourings defines a positive definite function on $F$ we introduce the following definition.
Let $g=g(T_+,T_-)\in F$. The associated link may be decomposed in the two upper and lower halves, that is $L(T_+,T_-)=(L_+(T_+),L_-(T_-))$. We may call these halves  as \emph{semi-links} \footnote{They are not links because they are not closed.}.

We may represent  these  semi-links associated to $(T_+,T_-)$ by the following tangles
$$
\includegraphics[scale=.5]{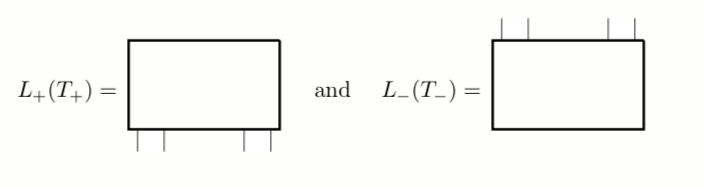}
$$
Consider a pair of bifurcating trees $(T_+,T_-)$ with $n$-leaves. The link $L(T_+,T_-)$ may be seen as the \emph{n-sum of the semi-links} which is defined by joining the corresponding strings in the rectangles. We denote the $n$-sum of semi-links by the symbol $\#_n$. 
\begin{lemma}
	Consider  $T_+, T_-\in\T_n$ and let $T'_+$, $T'_-$ be the trees obtained by adding a pair of opposite carets. Then 
	$$ 
		{\rm Col}_Q(L(T'_+,T'_+))=Q \cdot {\rm Col}_Q(L(T_+,T_+)).
	$$
\end{lemma}
\begin{proof}
The thesis follows from the fact that by \emph{Proposition} \ref{add-opp-car} we know that $L(T'_+,T'_+)=O\cup L(T_+,T_+)$. 
\end{proof}

Consider the colouring function
$$
	{\rm Col}_Q(g)\doteq {\rm Col}_Q(L(T_1,T_2)) Q^{-n}
$$
where $T_1, T_2\in \T_n$ and $g=g(T_1,T_2)\in F$. The above lemma implies that this function is well defined.

We want to introduce a partition function for the $Q$-colorings and we will recall some results of de la Harpe and Jones (see \cite{DLH-J}, p. 219). A link diagram is $4$-regular graph, where the vertices are the crossings. For each vertex $x$ we denote by $(a_1^x,a_2^x)$ the upper string and by  $(a_1^{-x},a_2^{-x})$ the lower string. We set
\begin{eqnarray*}
		  w(s^+,s^-,t^+,t^-)\doteq \left\{
 		\begin{array}{ll}
 		1 & \textrm{ if $s^+=t^+$ and $s^++t^+=s^-+t^-$}  \\ 
 		0   & \textrm{ otherwise.}
 		\end{array}
 		\right.	
\end{eqnarray*}
Consider the partition function defined by 
$$
	Z_{G}(Q)\doteq\sum_\tau \prod_{x\in V(G)} w(\tau(a_1^x),\tau(a_1^{-x}),\tau(a_2^{x}),\tau(a_2^{-x})) \ ,
$$
where the sum over the functions $\tau: E(G)\to \mathbb{Z}_Q$. Any such function is called \emph{state} of the graph $G$. 
We have that ${\rm Col}_Q(L)=Z_{L}(Q)$ (\cite{P}, \emph{Example} 4.1, p. 22).
 
\begin{theorem}
The function ${\rm Col}_Q(g)$, where $Q$ is any positive odd integer, is of positive type on $F$.
\end{theorem}
\begin{proof}
We can suppose that the $g_i=g(T_+^i,T_-^i)$  with $T_\pm^i\in \T_n$ for $i=1, \cdots, r$ and $g_ig_j^{-1}=g(T^i_+,T^j_+)$. Therefore, it is enough to consider  $\Big({\rm Col}_Q(L(T^i_+,T^j_+)) /Q^{n}\Big)_{i,j=1}^r$ and prove that $({\rm Col}_Q(L(T^i_+,T^j_+)))_{i,j=1}^r$ is positive semidefinite. Our idea is to use the partition function in order to prove the claim. 
Therefore
\begin{eqnarray*}
	{\rm Col}_Q(L(T^i_+,T^j_+)) &=& Z_{(L(T_+, T_-))}=\\
	&=& \sum_\tau \prod_{x\in V(L(T_+, T_-))} w(\tau(a_1^x),\tau(a_1^{-x}),\tau(a_2^{x}),\tau(a_2^{-x}))=\\
	&=&  \sum_{\tau_0} \left(\sum_{\tau_+} \prod_{x\in V(\vec{L}(T_+))} w(\tau(a_1^x),\tau(a_1^{-x}),\tau(a_2^{x}),\tau(a_2^{-x}))\right) \\
	& & \times   \left(\sum_{\tau_-} \prod_{x\in V(\vec{L}(T_-))} w(\tau(a_1^x),\tau(a_1^{-x}),\tau(a_2^{x}),\tau(a_2^{-x})) \right)\\
\end{eqnarray*}
where we have decomposed each state $\tau$ as $(\tau_0,\tau_+,\tau_-)$, $\tau_0$ being a function on edges in common between the semi-links, $\tau_+$ a function on the the remaining edges in $\vec{L}(T_+)$ and $\tau_-$ a function on the the remaining edges in $\vec{L}(T_-)$.
For any $\tau_0=(\tau_1,\cdots ,\tau_{2n})$, the expression $\sum_{\tau_+} \prod_{x\in V(\vec{L}(T_+))} w(\tau(a_1^x),\tau(a_1^{-x}),\tau(a_2^{x}),\tau(a_2^{-x}))$ defines the $\tau_0$-th component of a vector $v_{T_+}$ in $\H=\mathbb{C}^{3^{2n}}$. The choice of the vector space $\mathbb{C}^{3}$ is due to the three different values taken by the function $w(\cdot)$. 
Thus, we may define a vector $v_{T_+^i}$ and similarly a vector $v_{T_+^j}$ such that 
$$
	Z_{(L(T_+, T_-))}=\langle v_{T_+^i},v_{T_+^j}\rangle.
$$
It follows by \emph{Lemma} \ref{inner-prod-matrix} that the matrix $({\rm Col}_Q(L(T^i_+,T^j_+)))_{i,j=1}^r$ is positive semidefinite for any $r$, i.e. the function ${\rm Col}_Q(g)$ is of positive type.
\end{proof}

It is well-known that a  positive type function on a discrete group $G$ gives rise to a completely positive multiplier on the reduced $C^*$-algebra $C^*_r(G)$. 
One might investigate whether these or similar methods can be  exploited to say anything of interest  for other kind of multipliers on $F$.

\bigskip
{\parindent=0pt 

\smallskip Valeriano Aiello\\ 
Section de Math\'ematiques \\
Universit\'e de Gen\`eve \\
2-4 rue du Li\`evre, Case Postale 64, 1211 Gen\`eve 4, Switzerland\\
E-mail: valerianoaiello@gmail.com\\

\smallskip \noindent
Roberto Conti\\
Dipartimento di Scienze di Base e Applicate per l'Ingegneria \\
Sapienza Universit\`a di Roma \\
Via A. Scarpa 16, I-00161 Roma, Italy
\\ E-mail: roberto.conti@sbai.uniroma1.it
\par}

\end{document}